\documentclass[12pt]{article}
\title{Representation of $2\times 2$ matrices over perfect fields by a diagonal quadratic form}

\author{
  Murtuza Nullwala\\
  \small \text{Ramrao Adik Institute of Technology,}\\
  \small \text{Nerul, 400706.}
  \and
  Anuradha S. Garge\\
  \small \text{University of Mumbai,}\\
  \small \text{Mumbai, 400098.}
}

\date{} 
\usepackage{amsmath,amssymb,amsfonts}
\usepackage{amsthm} 

\theoremstyle{plain}
\theoremstyle{definition}
\newtheorem{T}{Theorem}
\newtheorem{D}[T]{Definition}

\newtheorem{N}[T]{Note}

\newtheorem{Lem}[T]{Lemma}
\newtheorem{Exa}[T]{Example}
\newtheorem{Rem}[T]{Remark}

\numberwithin{equation}{T}

\begin{document}

\maketitle

\begin{abstract}
  Let $F$ be a perfect field. In this paper, we show that a diagonal quadratic form
  $\sum_{i=1}^{m}a_iX_i^2$ over $F$ is universal over $M_2(F)$ if and only if atleast two of $a_1, a_2,\ldots, a_m$ are
  non-zero. 
\end{abstract}

\vspace*{0.1in}

\section{Introduction}
Let $R$ denote a commutative ring with unity. Let $M_n(R)$ denote the set of all $n\times n$ matrices over $R$. An $n-$ary quadratic form over $R$ is a homogeneous polynomial $q$ of degree $2$ in $n$ variables. Thus $q$ is of the form:
\begin{center}
$q(X_1,X_2,\ldots,X_n)=\sum_{i,j=1}^{n}a_{ij}X_iX_j$, where $a_{ij}\in R$.
\end{center}

The study of representation of all positive integers by a quadratic form is a classical one. In 1770, Lagrange was the first mathematician to show that any positive integer can be represented by the quadratic form
\begin{equation}
X_1^2+X_2^2+X_3^2+X_4^2.
\end{equation}
In 1916, S. Ramanujan enlisted $54$ (see \cite{Manjul}) diagonal forms which represent all positive integers. If $q(X_1,X_2,\ldots,X_n)=\sum_{i,j=1}^{n}a_{ij}X_iX_j\in \mathbb{Z}[X_1,X_2,\ldots,X_n]$ is a quadratic form then we can associate a symmetric matrix $M_q$ of order n, defined as $M_q=(m_{ij})$ where $m_{ij}=\dfrac{1}{2}(a_{ij}+a_{ji})$. 
\begin{D}
(I) A quadratic form $q$ over $\mathbb{Z}$ is said to have integer matrix $M_q$ if all entries of $M_q$ are integers.\\ 
(II) A quadratic form is said to be positive definite if $q(x_1,x_2,\ldots,x_n)>0$, $x_i\in\mathbb{Z}$ and atleast one $x_i\neq 0$ for some $i,\;1\leq i\leq n$. 
\end{D}

\begin{D}
  A diagonal quadratic form $\sum_{i=1}^{m}a_iX_i^2$, where $a_i\in F$ is said to be universal over $F$ if it represents every
  element of $F.$
\end{D}

In 2000, Manjul Bhargava (see \cite{Manjul}) proved an interesting criterion called the fifteen theorem, to decide when a positive definite quadratic form with integer matrix represents all positive integers. We recall this theorem below:
\begin{T}
If a positive definite quadratic form having an associated integer matrix represents every positive integer up to fifteen, then it represents every positive integer.
\end{T}
In 1986, Vaserstein (see \cite{Vas}) proved that for $n \geq 2$, every $n\times n$ integral matrix is a
sum of three squares i.e., the quadratic form $X_1^2 + X_2^2 + X_3^2$ is universal over $M_n(\mathbb{Z}).$
In 2017, Jungin Lee (see \cite{Lee}) gave a criterion to check when a diagonal quadratic form $\sum_{i=1}^{n}a_iX_i^2$, $a_i\in \mathbb{Z}$ represents all matrices in $M_2(\mathbb{Z})$ i.e. a criterion to check when $\sum_{i=1}^{n}a_iX_i^2$ is universal over
$M_2(\mathbb{Z}).$ We recall below his theorem:

\begin{T}
  A quadratic form $\sum_{i=1}^{m}a_iX_i^2,$ $a_i \in \mathbb{Z}$ is universal over $M_2(\mathbb{Z})$ if and only
  if there is no prime which divides $m-1$ numbers of $a_1,a_2,\ldots, a_m$ and there exist three numbers of $a_1,a_2,\ldots,a_m$ which are not multiples of $4$.
\end{T}

We now consider quadratic forms over a field with characteristic not equal to $2$. Then to every quadratic form $q$ we can still associate a symmetric matrix $M_q$. One has the following result for the universality of the quadratic form $q$ (see \cite{Lam}):

\begin{T}
   Let $F$ be a field of characteristic not equal to $2$.
   If for a quadratic form $q$ over $F,$ $q(x_1, x_2, \ldots, x_n)= 0$ for some
   $x_1, x_2,\ldots, x_n \in F$ with atleast one $x_i\neq 0$ and the determinant of the associated matrix $M_q$
   is non-zero then $q$ is universal over $F.$ 
\end{T}

In the light of the paper by Jungin Lee, S. A. Katre posed the following question: given a diagonal quadratic form
$a_1X_1^2+a_2X_2^2+\cdots +a_mX_m^2$ with $a_i \in F,$ when is it universal over $M_n(F),$ $n \geq 2$?
In this paper, we answer the above question for $n = 2$ by proving the following result: 

\noindent {\bf{Theorem:}}
A diagonal quadratic form $\sum_{i=1}^{m}a_iX_i^2$ over a perfect field $F$ is universal over $M_2(F)$ if and only if atleast
two of $a_1, a_2, \ldots, a_m$ are nonzero.

The advantage of the above result is that one does not have to go to the matrix associated with the quadratic form. From the coefficients of the quadratic form itself, we can easily check the universality.

\section{Proof of the main result}
We prove our main result by proving the following sequence of theorems.

\begin{T} \label{T2}
A quadratic form $aX^2$ where $a\in F$ is not universal over $M_2(F)$.
\end{T}
\begin{proof}
If $a=0$ then clearly it is not universal over $M_2(F)$.  So we suppose that $a\neq 0$. Consider 
$a\left[\begin{array}{cc} x & y \\ z & w \end{array} \right]^2=\left[\begin{array}{cc} 0 & 1 \\ 0 & 0 \end{array} \right]$. Then we get the following system of equations
\begin{eqnarray}
a(x^2+yz)=0 \label{eq:201} \\
ay(x+w)=1 \label{eq:202} \\
az(x+w)=0 \label{eq:203} \\
a(yz+w^2)=0. \label{eq:204} 
\end{eqnarray}
From Equation (\ref{eq:202}) and (\ref{eq:203}), we get $z=0$. Then from Equation (\ref{eq:201}) and (\ref{eq:204}), we get $x=0$ and $w=0$, a contradiction to Equation (\ref{eq:202}).
\end{proof}

\begin{T} \label{T3}
A quadratic form $a_1X_1^2+a_2X_2^2$ where $a_1,a_2\in F$, char $F\neq 2$ and both $a_1,a_2\neq 0$, is universal over $M_2(F)$.
\end{T}
\begin{proof}
Consider $a_1\left[\begin{array}{cc} x_1 & y_1 \\ z_1 & w_1 \end{array} \right]^2+a_2\left[\begin{array}{cc} x_2 & y_2 \\ z_2 & w_2 \end{array} \right]^2=\left[\begin{array}{cc} p & q \\ r & s \end{array} \right]$\\
which gives us the following system of equations
\begin{eqnarray}
a_1x_1^2+a_2x_2^2+a_1y_1z_1+a_2y_2z_2=p \label{eq:1} \\
a_1y_1(x_1+w_1)+a_2y_2(x_2+w_2)=q \label{eq:2} \\
a_1z_1(x_1+w_1)+a_2z_2(x_2+w_2)=r \label{eq:3} \\
a_1w_1^2+a_2w_2^2+a_1y_1z_1+a_2y_2z_2=s. \label{eq:4}
\end{eqnarray}
From Equation (\ref{eq:1}) and (\ref{eq:4}), we get 
\begin{equation}\label{eq:8}
a_1y_1z_1+a_2y_2z_2=p-(a_1x_1^2+a_2x_2^2)=s-(a_1w_1^2+a_2w_2^2).
\end{equation}
We will find a solution to the equation 
\begin{equation}
p-(a_1x_1^2+a_2x_2^2)=s-(a_1w_1^2+a_2w_2^2)\label{eq:100}
\end{equation}
such that $x_1+w_1 \neq 0.$
If we do so, as $a_1\neq 0$ and $a_2\neq 0$, from Equations (\ref{eq:2}) and (\ref{eq:3}), we get the values of $y_1$ and $z_1$ as 
\begin{equation}
y_1=\dfrac{q-a_2y_2(x_2+w_2)}{a_1(x_1+w_1)} \label{eq:5}
\end{equation}
\begin{equation}
z_1=\dfrac{r-a_2z_2(x_2+w_2)}{a_1(x_1+w_1)}. \label{eq:6}
\end{equation}
Now consider Equation (\ref{eq:100}) in the form $a_1x_1^2+a_2x_2^2-a_1w_1^2-a_2w_2^2=p-s$ and choose $x_2 = w_2 = 0.$
Thus, our equation now reads $a_1x_1^2 - a_1w_1^2 = p-s.$ 
{\rm Case}(I): If $p-s=0$, as $a_1\neq 0$ take $x_1=w_1$. Since char $F\neq 2$ by selecting any nonzero value for $x_1$
we have $x_1+w_1\neq 0.$

{\rm Case}(II): If $p-s\neq 0$, we get $x_1^2-w_1^2=\dfrac{p-s}{a_1}$ i.e. $(x_1-w_1)(x_1+w_1)=\dfrac{p-s}{a_1}$.
We shall choose $x_1$,$w_1$ such that $x_1-w_1=\dfrac{p-s}{a_1}$ and $x_1+w_1=1.$ To see this, adding the two equations
we get $2x_1=\dfrac{p-s+a_1}{a_1}$ i.e. $x_1=\dfrac{p-s+a_1}{2a_1}$ and $w_1=1-x_1=\dfrac{a_1-(p-s)}{2a_1}$.\\
From each of the cases above, we have $x_1+w_1\neq 0$. Now from Equation (\ref{eq:5}) and (\ref{eq:6}),
since we have chosen $x_2 = w_2 = 0,$ we have the following values of $y_1$ and $z_1:$ 
\begin{eqnarray}
y_1=\dfrac{q}{a_1(x_1+w_1)} \label{eq:9}\\
z_1=\dfrac{r}{a_1(x_1+w_1)}. \label{eq:10}
\end{eqnarray}
Now from Equation (\ref{eq:8}), we get 
\begin{equation}\label{eq:111}
a_2y_2z_2=-a_1y_1z_1+s-a_1w_1^2
\end{equation}
\begin{equation}\label{eq:112}
\text{i.e.} \; y_2z_2=\dfrac{1}{a_2}(-a_1y_1z_1+s-a_1w_1^2).
\end{equation}
Then by fixing any non-zero value of $z_2$ we get the value of $y_2$. Hence, we get a solution to the system of equations.
\end{proof}

\begin{D}
A field $F$ is said to be perfect if it either has characteristic $0$ or when characteristic $p>0$ then every element of $F$ is a $p^{th}$ power i.e. $F^p=F$ where $F^p=\{x^p:x\in F\}$.
\end{D}

\begin{Lem} \label{L1}
Let $F$ be a perfect field of characteristic $2.$ 
Then the map $f:F\rightarrow F$ defined by $f(x)=x^2$ is an isomorphism.
\end{Lem}
\begin{proof}
  We have $f(x+y)=(x+y)^2=x^2+2xy+y^2=x^2+y^2=f(x)+f(y)$ and $f(xy)=x^2y^2=f(x)f(y)$. As $f$ is a non-zero homomorphism it implies that $f$ is injective (see \cite{Dummit}). Also $F$ is a perfect field of char $F=2$ which implies $F = F^2$ and hence $f$ is onto.
  Therefore, $f$ is an isomorphism.
\end{proof}

\begin{T} \label{T5}
A quadratic form $a_1X_1^2+a_2X_2^2$ where $a_1,a_2\in F$, $F$ is a perfect field of char $F=2$ and both $a_1,a_2\neq 0$ is universal over $M_2(F)$.
\end{T}
\begin{proof}
Consider the following system of equations
\begin{eqnarray}
a_1x_1^2+a_2x_2^2+a_1y_1z_1+a_2y_2z_2=p \label{eq:11} \\
a_1y_1(x_1+w_1)+a_2y_2(x_2+w_2)=q \label{eq:12} \\
a_1z_1(x_1+w_1)+a_2z_2(x_2+w_2)=r \label{eq:13} \\
a_1w_1^2+a_2w_2^2+a_1y_1z_1+a_2y_2z_2=s. \label{eq:14}
\end{eqnarray}
From Equations (\ref{eq:11}) and (\ref{eq:14}), we will find solution to the equation 
\begin{equation} \label{eq:15}
a_1x_1^2+a_2x_2^2-a_1w_1^2-a_2w_2^2=p-s.
\end{equation}
Since char $F=2$, we may consider the following equation
\begin{equation} \label{eq:16}
a_1x_1^2+a_2x_2^2+a_1w_1^2+a_2w_2^2=p+s.
\end{equation}
Case (I): If $p+s=0$ then take $x_2=w_1=0$, $w_2=1$ and then from Lemma \ref{L1} as $a_1\neq 0$ and $a_2\neq 0$, $a_1x_1^2=a_2$ has a non-zero solution.\\
Case (II): If $p+s\neq 0$ then take $x_2=w_1=w_2=0$ then from Lemma \ref{L1}, $a_1x_1^2=p+s$ has a non-zero solution in $F$ as $a_1\neq 0$ and $p+s\neq 0$.\\
From each of the above cases, we have $x_1+w_1\neq 0$ and further following the proof of Theorem \ref{T3}, we have a solution to the required system of equations.
\end{proof}

\begin{T} \label{T6}
A diagonal quadratic form $\sum_{i=1}^{m}a_iX_i^2$ over a perfect field $F$ is universal over $M_2(F)$ if and only if atleast two elements of $a_1,a_2,\ldots,a_m$ are nonzero.
\end{T}
\begin{proof}
Without loss of generality let $a_1, a_2\neq 0$ and $X_i=0$ for $3\leq i \leq m$. Then from Theorem \ref{T3} and \ref{T5}, $\sum_{i=1}^{m}a_iX_i^2$ is universal. Conversely, suppose $\sum_{i=1}^{m}a_iX_i^2$ is universal then from Theorem \ref{T2}, we have the result.
\end{proof}

\begin{Rem}
  We demonstrate the importance of a perfect field (see \cite{Law}) by showing that the theorem may fail, if this condition does not hold. 
  Consider the field $F = \mathbb{Z}_2(X).$ This field is not perfect since $X$ is not a square in $\mathbb{Z}_2(X);$
  for suppose $X=\dfrac{(a_0+a_1X+a_2X^2+\cdots+X^k)^2}{(b_0+b_1X+b_2X^2+\cdots+X^t)^2}$ which implies that
  $X(b_0+b_1X+b_2X^2+\cdots+X^t)^2=(a_0+a_1X+a_2X^2+\cdots+X^k)^2$. By comparing degrees, we get $2t+1=2k$, a contradiction. 
  We now prove that the quadratic form $X_1^2+X_2^2$ is not universal over $M_2(\mathbb{Z}_2(X)).$
  To see this, consider 
\begin{center}
  $\left[\begin{array}{cc} x_1 & y_1 \\ z_1 & w_1 \end{array} \right]^2+\left[\begin{array}{cc} x_2 & y_2 \\ z_2 & w_2
    \end{array} \right]^2=\left[\begin{array}{cc} p & q \\ r & s \end{array} \right]$
\end{center}
which gives us the following system of equations: 
\begin{eqnarray}
x_1^2+x_2^2+y_1z_1+y_2z_2=p \label{eq:19} \\
y_1(x_1+w_1)+y_2(x_2+w_2)=q \label{eq:20} \\
z_1(x_1+w_1)+z_2(x_2+w_2)=r \label{eq:21} \\
w_1^2+w_2^2+y_1z_1+y_2z_2=s. \label{eq:22}
\end{eqnarray}
Existence of solution to the equation 
\begin{equation}\label{eq:23}
x_1^2+x_2^2+w_1^2+w_2^2=p+s
\end{equation}
is necessary for the existence of a solution to the original system of equations. Since Char $F=2$, Equation (\ref{eq:23}) can be written as 
\begin{equation} \label{eq:24}
(x_1+x_2+w_1+w_2)^2=p+s.
\end{equation}
Since $\mathbb{Z}_2(X)$ is not a perfect field the equation $x^2=a$ need not have a solution for every $a \in \mathbb{Z}_2(X).$
(For example, $a = X$ works as shown above.) This proves that $X_1^2+X_2^2$ is not universal over $M_2(\mathbb{Z}_2(X)).$  
\end{Rem}

\section{Algorithm of the proof}
For a given matrix $A$ over $\mathbb{Q}$ and $a_1,a_2\in\mathbb{Q}$, following is the algorithm to find entries of the matrices $X_1$ and $X_2$ such that $a_1X_1^2+a_2X_2^2=A$ and the program can be implemented in Python language.\\\\
1.	Import the packages numpy and fraction for declaring and displaying the array values. \\
2.	Declare a two dimensional array A \\
3.	p=A[0][0] \\
4.	q=A[0][1] \\
5.	r=A[1][0] \\
6.	s=A[1][1] \\
7.	Initialize the values a1, a2. \\
8.	Initialize x2=0, w2=0 and z2=1. \\
9.	if p-s==0: \\
10.	   x1=w1=1 \\
11.	   y1=q/(a1*(x1+w1)) \\
12.	   z1=r/(a1*(x1+w1)) \\
13.	   y2=(s-a1*y1*z1-a1*pow(w1,2))/(a2) \\
14.	Display y1,y2 and z1 \\
15.	else: \\
16.	    x1=(p-s+a1)/(2*a1) \\
17.	    w1=(a1-p+s)/(2*a1) \\
18.	    y1=q/(a1*(x1+w1)) \\
19.	    z1=r/(a1*(x1+w1)) \\
20.	    y2=(s-a1*y1*z1-a1*pow(w1,2))/(a2) \\
21.	Display x1, w1, y1, y2, z1 \\

\begin{Exa}
Consider $F=\mathbb{Q}$ and a diagonal quadratic form $q(X_1,X_2)=2X_1^2+X_2^2$. We have $\text{char}\;F=0$. Let $A=\left[\begin{array}{cc} \frac{1}{5} & 2 \\ 0 & -1 \end{array}\right]\in M_2(\mathbb{Q})$. By following the steps of Theorem \ref{T3}, we will find a solution to equation
\begin{equation}
2\left[\begin{array}{cc} x_1 & y_1 \\ z_1 & w_1 \end{array}\right]^2+\left[\begin{array}{cc} x_2 & y_2 \\ z_2 & w_2 \end{array}\right]^2=\left[\begin{array}{cc} \frac{1}{5} & 2 \\ 0 & -1 \end{array}\right].
\end{equation}
Let $a_1=2,a_2=1,p=\dfrac{1}{5},q=2,r=0,s=-1$. Take $x_2=w_2=0$ and $z_2=1$. We have $p-s=\dfrac{6}{5} \neq 0$. We get $x_1=\dfrac{p-s+a_1}{2a_1}=\dfrac{4}{5}$ and $w_1=\dfrac{a_1-(p-s)}{2a_1}=\dfrac{1}{5}$. Then we get $y_1=\dfrac{q}{a_1(x_1+w_1)}=1$, $z_1=\dfrac{r}{a_1(x_1+w_1)}=0$ and $y_2=\dfrac{1}{a_2}(-a_1y_1z_1+s-a_1w_1^2)=\dfrac{-27}{25}$.
\end{Exa}

\medskip

\noindent
    {\bf Acknowledgement:} We sincerely thank Professor S. Kotyada for useful discussions. We also thank Professor S. A. Katre
    for all his support and encouragement and Ms. Nivedita Elanshekar for helping in computational part.

{\it ~ Email Address}:~murtuza.ibrahim@gmail.com,~anuradha.garge@gmail.com

\end{document}